\documentclass[11pt]{amsart}
\usepackage{amsmath, amssymb, amsbsy, amsfonts, amsthm, latexsym, amsopn, amstext, amsxtra, euscript, amscd, color, mathrsfs}
\usepackage[normalem]{ulem}
\usepackage{soul}
\usepackage{cite}

\PassOptionsToPackage{hyphens}{url}\usepackage{hyperref}
 
 \usepackage[capbesideposition=outside,capbesidesep=quad]{floatrow}

\restylefloat{table}
\restylefloat{table}
            
\usepackage{multirow,caption}
            
\usepackage{amscd}
\usepackage{color,enumerate}
\newcommand{\RNum}[1]{\lowercase\expandafter{\romannumeral #1\relax}}

\usepackage[colorinlistoftodos,prependcaption,textsize=tiny]{todonotes}

\newtheorem{thm}{Theorem}[section]
\newtheorem{lem}[thm]{Lemma}

\newtheorem{conj}[thm]{Conjecture}

\newtheorem{thm-con}[thm]{Theorem-Conjecture}
\numberwithin{equation}{section}

\theoremstyle{definition}

\newcommand{\F}{\mathbb F}

\def\Tr{{\rm Tr}}

\usepackage{todonotes}
\usepackage{lipsum}
\usepackage{marginnote}

\begin{document}
\title[Permutation Trinomials]{\Large A proof of a conjecture on permutation trinomials}
\author{Daniele Bartoli}
\address{Department of Mathematics and Computer Science, University of Perugia,  06123 Perugia, Italy}
\email{daniele.bartoli@unipg.it}
\author{Mohit Pal}
\address{Department of Informatics, University of Bergen, PB 7803, N-5020,
Bergen, Norway}
\email{Mohit.Pal@uib.no}
\author{Pantelimon St\u anic\u a}
\address{Applied Mathematics Department, Naval Postgraduate School,
Monterey, CA 93943, USA}
\email{pstanica@nps.edu}

\date{\today}


\begin{abstract}
In this paper we use algebraic curves and other algebraic number theory methods to show the validity of a permutation polynomial conjecture regarding $f(X)=X^{q(p-1)+1} +\alpha X^{pq}+X^{q+p-1}$, on finite fields $\F_{q^2}, q=p^k$, from [A. Rai, R. Gupta, {\it Further results on a class of permutation trinomials}, Cryptogr. Commun. 15 (2023),   811--820]. 
\end{abstract}

\maketitle
 
\section{Introduction}

Let $\F_q$ be the finite field with $q=p^k$ elements. We denote by $\F_q^*$ the multiplicative group of nonzero elements of $\F_q$ and by $\F_q[X]$ the polynomial ring in the indeterminate $X$ over a finite field $\F_q$. A polynomial $f \in \F_q[X]$ is called a permutation polynomial if the equation $f(X)=a$ has exactly one solution in $\F_q$ for each $a\in \F_q$.

Permutation polynomials with a few terms are of great importance due to their applications in cryptography and coding theory. The classification of the simplest kind of permutation polynomials, i.e., monomials is well known. However, a complete classification of permutation binomials and trinomials is not yet known. A lot of research has been done in this direction in recent years and the reader may refer to~\cite{HouBT}, and the references therein. 

Permutation trinomials with Niho exponents over $\F_{q^2}$, that is, $F(x)=x+\alpha x^{s_1(q-1)+1}+\beta x^{s_2(q-1)+1},$ where $s_1,s_2$ are integers and coefficients $\alpha,\beta\in\{-1,1\}$ have been studied extensively, and the reader can easily find many references on these polynomials.
An interesting problem in this direction is to determine the permutation behaviour of  polynomials of the form
\begin{equation}\label{palpha}
f_{\alpha, \beta}(X)= X^{q(p-1)+1}+ \alpha X^{pq}+\beta X^{q+p-1} 
\end{equation}
over $\F_{q^2}$, where $q=p^k$ for some positive integer $k$ and odd prime $p$; and $\alpha\beta \in \F_q^*$. 
In~\cite{Houp3}, Hou completely characterized the permutation behaviour  of the polynomial $f_{\alpha, \beta}$ when the characteristic of the underlying finite field is $3$. Bai and Xia~\cite{BX18} looked at the permutation property of the polynomial $f_{1,-1}$ over $\F_{q^2}$, $q=p^k$, for $p=3,5$, showing that it is a permutation  if and only if $k$ is even.
In~\cite{GR_JAA22}, Gupta and Rai considered the permutation behavior of $f_{\alpha,1}$ when the characteristic of the underlying finite field is $5$ and showed that when $k>1$ then $f_{\alpha,1}$ is a permutation polynomial if and only if $\alpha= -1$ and $k$ is even. In a subsequent article~\cite{rai} the authors showed that in the case of characteristic $7$, $f_{\alpha, 1}$ is a permutation polynomial if and only if either $\alpha=-3$ and $k=1$; or $\alpha=-1$ and $k=2$. Moreover, the authors showed that when $p>3$ and $k=1$ then $f_{\alpha, 1}$ is a permutation polynomial if and only if $\alpha=-3$. For $p>7$ and $k>1$, the authors proposed the following conjecture.

\begin{conj}
\label{Conj}
Let $q=p^k$, where $p>7$ is a prime, $k>1$. Then, for $\alpha \in \F_q^*$ and $k>1$, the trinomial
\[
f(X)=X^{q(p-1)+1}+ \alpha X^{pq}+X^{q+p-1} 
\]
is a permutation polynomial over $\F_{q^2}$ if and only if $\alpha = -1$ and $k=2$. 
\end{conj}
It is the intent of our paper to completely prove this conjecture. 


\section{Preliminaries}
We now recall some basic facts on curves/surfaces over (finite) fields. For more details, we refer  to~\cite{Ha77,HKT13}, or the reader's favorite algebraic geometry book. As customary, for a field $\F$, we denote by $\overline{\F}$ its algebraic closure, and by $\mathbb{P}^m(\F)$ (respectively, $\mathbb{A}^m(\F)$)  the $m$-dimensional projective (respectively, affine) space over the field $\F$. Solutions of polynomial equations (systems) form what we call algebraic hypersurfaces (varieties). An algebraic hypersurface defined over a field $\mathbb{F}$ is called absolutely irreducible if the associated polynomial is irreducible over every algebraic extension of $\mathbb{F}$. An absolutely irreducible $\mathbb{F}$-rational component of a hypersurface defined by a polynomial $F$ is  an absolutely irreducible hypersurface, associated to a factor of $F$ defined over $\overline{\F}$.

In two dimensions, $\mathcal{C}$ is an affine curve over a field $\F$ if it is the zero set of a polynomial $F(X,Y)\in\F[X,Y]$. A projective curve $\mathcal{C}$  over a field $\F$ is the zero set of a homogeneous polynomial $F(X,Y,Z)\in\F[X,Y,Z]$.  The polynomial $F$ is the defining polynomial of $\mathcal{C}$.

Given two plane curves $\mathcal{A}$ and $\mathcal{B}$ and a point $P$ on the plane, the \emph{intersection number} $I(P, \mathcal{A} \cap \mathcal{B})$ of $\mathcal{A}$ and $\mathcal{B}$ at the point $P$ can be defined by seven axioms. We do not include its precise and long definition here. For more details, we refer to~\cite{MR1042981} and~\cite{HKT13} where the intersection number is defined equivalently in terms of local rings and in terms of resultants, respectively. Intuitively, the intersection number counts the number of times the two curves intersect at $P$.

Consider an affine point $P=(\alpha,\beta)$ of a plane curve $\mathcal{C}: F(X,Y)=0$.
We write 
\begin{equation}\label{Eq:multiplicty}
F(X+\alpha,Y+\beta)=F_0(X,Y)+F_{1}(X,Y)+\cdots +F_d(X,Y),
\end{equation}
where each $F_i(X,Y)$ is either the 0-polynomial or homogeneous of degree $i$. The multiplicity  $m_P(\mathcal{C})$ of $\mathcal{C}$ at $P$ is defined as the smallest integer $i$ such that $F_i(X,Y)\neq 0$ in~\eqref{Eq:multiplicty}. A singular point of multiplicity $m$ is said to be an $m$-fold singular point. For an  $m$-fold singular point $P=(\alpha,\beta)$, let $\ell(X,Y)$ be a linear factor of $F_m(X,Y)$ in~\eqref{Eq:multiplicty}. Then $\ell(X-\alpha,Y-\beta)=0$ is called a tangent line of $\mathcal{C}$ at $P$ and it intersects $\mathcal{C}$ with multiplicity larger than $m$ at $P$. An $m$-fold singular point is said to be ordinary if all the tangent lines at it are distinct (i.e. $F_m(X,Y)$ is separable).

Concerning the intersection number of two curves at a point, the following  classical result  can be found in most textbooks on algebraic curves.
\begin{lem}\label{le:ordinary_singular}
 	Let $\mathcal{A}$ and $\mathcal{B}$ be two plane curves.  For any affine point $P$, the intersection number satisfies the inequality
 	\[ I(P, \mathcal{A}\cap \mathcal{B})\ge m_P(\mathcal{A}) m_P(\mathcal{B}), \]
	with equality if and only if the tangents at $P$ to $\mathcal{A}$ are all distinct from the tangents at $P$ to $\mathcal{B}$.
 \end{lem}

We now recall B\'ezout Theorem~\cite[Theorem 3.13]{HKT13}, which will be used later. 
\begin{thm}[B\'ezout Theorem]
Let $\mathcal{C}_1,\mathcal{C}_2$ be two projective plane curves of degrees $d_1$, respectively, $d_2$. If $\mathcal{C}_1$ and  $\mathcal{C}_2$ do not have a common component, then the sum of the intersection multiplicities at their common points is
\[
\sum_{P\in\mathcal{C}_1\cap \mathcal{C}_2} I\left(P,\mathcal{C}_1\cap \mathcal{C}_2\right)=d_1d_2.
\]
\end{thm}

Finally, we will make use of the following version of the celebrated Hasse-Weil theorem. 
\begin{thm}[Aubry-Perret bound \textup{\cite[Corollary 2.5]{AubryPerret}}]
\label{Th:AubryPerret}
Let $\mathcal{C}\subset \mathbb{P}^{n}(\mathbb{F}_q)$ be an absolutely irreducible curve which is a complete intersection of $(n-1)$ hypersurfaces of degrees $d_1,\ldots, d_{n-1}$ and set $d=\prod_{i=1}^{n-1}d_i$. Then the number $\mathcal{C}(\mathbb{F}_q)$ of $\mathbb{F}_q$-rational points of $\mathcal{C}$ in $ \mathbb{P}^{n}(\mathbb{F}_q)$ satisfies

\vspace*{-0.5 cm}
\begin{equation}\label{EQ:HW3}
q+1-(d-1)(d-2)\sqrt{q}\leq \#\mathcal{C}(\mathbb{F}_q)\leq q+1+(d-1)(d-2)\sqrt{q}.
\end{equation}
\end{thm}



 We shall also use the following result of~\cite{CH04} (we make some small changes to fit our purpose).
\begin{lem}\textup{\cite[Theorem 3]{CH04}}
\label{L1}
 Let $p$ be a prime, $\ell$ be a non-negative integer, and the polynomial $F(X) = X^{p^n}-AX-B \in \F_{p^\ell}[X]$, $A \neq 0$. Let $d= \gcd(\ell,n)$, $m=\ell/d$ and $\Tr_d$ be the relative trace from $\F_{p^\ell}$ to $\F_{p^d}$. For $0\leq i \leq m-1$, define $\displaystyle t_i=\sum_{j=i}^{m-2} p^{n(j+1)}$. Put $\alpha_0 = A$ and $\beta_0=B$. If $m>1$, then for $1\leq r \leq m-1$, we let $\displaystyle \alpha_r 
 =A^{\frac{p^{n(r+1)}-1}{p^n-1}}$ and 
 $\displaystyle \beta_r = \sum_{i=0}^{r}A^{s_i} B^{p^{n i}}$ where $\displaystyle s_i = \sum_{j=i}^{r-1} p^{n(j+1)}$, for $0 \leq i \leq r-1$, and $s_r =0$.
 \begin{enumerate}[(i)]
  \item If $\alpha_{m-1} = 1$ and $\beta_{m-1} \neq 0$ then the trinomial $F$ has no roots in $\F_{p^n}$.
  \item If $\alpha_{m-1}\neq 1$ then $F$ has a unique root, namely $x = \dfrac{\beta_{m-1}}{1-\alpha_{m-1}}$. 
  \item If $\alpha_{m-1} = 1, \beta_{m-1} = 0$, $F$ has $p^d$ roots in $\F_{p^\ell}$  given by $x+\delta\tau$, where $\delta \in \F_{p^d}$, $\tau$ is fixed in $\F_{p^\ell}$ with $\tau^{p^n-1} = A$ (that is, a $(p^n-1)$-root of $A$), and, for any $c \in \F_{p^\ell}^{*}$, satisfying $\Tr_d (c)\neq 0$ then $\displaystyle x = \frac{1}{\Tr_d (c)} \sum_{i=0}^{m-1} \left(\sum_{j=0}^{i} c^{p^{nj}}\right) A^{t_i} B^{p^{ni}}$. 
 \end{enumerate}
 \end{lem} 
\section{The proof of the conjecture for \texorpdfstring{$k\geq4$}{}}

We consider first $k\geq 4$. Consider again the polynomial 
\[
f(X)= X^{(p-1)q+1}+\alpha X^{pq}+ X^{q+p-1} \in \mathbb{F}_{q^2}[X],
\]
where $\alpha \in \F_q^*$ and $q=p^k$, $k\geq 4$, $p$ an odd prime. It is well-known~\cite{MR1812309,MR2495253,AKBARY201151} that $f(X)=X^{q+p-1}(X^{(q-1)(p-2)}+\alpha X^{(q-1)(p-1)} +1)$ permutes $\mathbb{F}_{q^2}$, since $\gcd(q+p-1,q^2-1)=1$, if and only if 
\[
g_{\alpha}(X)=X^{q+p-1}(X^{p-2}+\alpha X^{p-1}+1)^{q-1}
\] 
permutes $\mu_{q+1}=\{a \in \mathbb{F}_{q^2} \ : \ a^{q+1}=1\}$. We can restrict our investigation to those $\alpha$ such that $\alpha+2\neq 0$, otherwise $g_{\alpha}(1)=0$, and so, it cannot be a permutation on $\mu_{q+1}$. For any $x\in \mu_{q+1}$,
\begin{eqnarray*}
    g_{\alpha}(x)&=&x^{p-2}\frac{(x^{p-2}+\alpha x^{p-1}+1)^{q}}{x^{p-2}+\alpha x^{p-1}+1}\\
    &=&x^{p-2}\frac{(1/x)^{p-2}+\alpha (1/x)^{p-1}+1}{x^{p-2}+\alpha x^{p-1}+1}\\
    &=&\frac{x+\alpha + x^{p-1}}{x^{p-1}+\alpha x^{p}+ x}.
\end{eqnarray*}

We shall need below the well-known fact that $\mu_{q+1}\setminus \{1\}=\{(t+i)/(t-i) : t \in \mathbb{F}_q, i^q=-i\}$. Consider 
\begin{eqnarray*}
    F_{\alpha}(X,Y)&:=&(X+\alpha + X^{p-1})(Y^{p-1}+\alpha Y^{p}+ Y)- (Y+\alpha + Y^{p-1})(X^{p-1}+\alpha X^{p}+ X)\\
    &=&\alpha( X^{p-1}Y^p-X^{p}Y^{p-1}+XY^{p}-X^{p}Y+\alpha (Y-X)^p+Y^{p-1}-X^{p-1}+Y-X).
\end{eqnarray*}
It is readily seen that $g_{\alpha}$ permutes $\mu_{q+1}$ if and only if there exist no pairs $(x,y) \in \mu_{q+1}^2$, $x\neq y$, such that  $F_{\alpha}(x,y)=0$. The polynomial $F^{(1)}_{\alpha}(X,Y):= F_{\alpha}(X,Y)/(X-Y)$ defines a curve $\mathcal{C}_{\alpha}$ in $\mathbb{A}^2(\mathbb{F}_{q^2})$ that is $\mathbb{F}_{q^2}$-birationally equivalent to the curve $\mathcal{D}_{\alpha}\subset \mathbb{A}^2(\mathbb{F}_{q})$ defined by 
\[
G_{\alpha}(X,Y):= \frac{(X-i)(Y-i)}{2i(Y-X)}F_{\alpha}\left(\frac{X+i}{X-i},\frac{Y+i}{Y-i}\right).
\]
Such a birationality does not preserve the $\mathbb{F}_{q}$-rationality of points nor of components of the two curves in general, but sends $(x,y) \in \mu_{q+1}^2$ in $\mathcal{C}_{\alpha}$ into $(\overline{x},\overline{y}) \in \mathbb{F}_{q}^2$ in $\mathcal{D}_{\alpha}$ and vice-versa and preserves the number of absolutely irreducible components of the two curves.
Thus, the curve $\mathcal{D}_{\alpha}$ is absolutely irreducible if and only if so is $\mathcal{C}_{\alpha}$.

We aim now at proving that the curve $\mathcal{C}_{\alpha}$ is absolutely irreducible. First we compute the set of singular points of $\mathcal{C}_{\alpha}$. There are two points at infinity, namely $P_{\infty}=(1:0:0)$ and $Q_{\infty}=(0:1:0)$ and tangent lines at these $(p-1)$-fold singular points are $Y=\eta$ with $\eta^{p-1}+\eta+\alpha=0$ and $X=\xi$ with $\xi^{p-1}+\xi+\alpha=0$. We are interested in bounding the multiplicity of intersection of two putative components passing through them. Consider the point $Q_{\infty}$. Let $F_{\alpha}(X,Y,T)$ be the homogenized companion polynomial of $F_{\alpha}(X,Y)$. Now
\begin{eqnarray*}
    F_{\alpha}(X,1,Y)&:=&\alpha( X^{p-1}-X^{p}+XY^{p-2}-X^{p}Y^{p-2}\\
    &&+\alpha (1-X^p)Y^{p-1}+Y^{p}-X^{p-1}Y^p+Y^{2p-2}-XY^{2p-2})
\end{eqnarray*}
and we need to investigate the point $(0:0:1)$ for the curve $F_{\alpha}(X,1,Y)=0$. The tangent lines are $X=\xi Y$ with $\xi^{p-1}+\xi+\alpha=0$. If they are all distinct then $(0:0:1)$ (and so $Q_{\infty}$) is an ordinary point and the multiplicity of intersection of two putative components through it is at most $(p-1)^2/4$ (see Lemma~\ref{le:ordinary_singular}). More precisely, this also shows that if the components have  $r_1$ and $r_2$ tangent lines through $(0:0:1)$ (respectively) then the multiplicity of intersection would be precisely $r_1r_2$.

Let $X=\xi Y$ be a repeated tangent at $(0:0:1)$. This implies $2\xi+\alpha=0$ and $\xi^{p-2}=1$. In particular $\xi\neq 1$ otherwise $\alpha=-2$, a contradiction. Since the homogeneous part of degree $p$ in $F_{\alpha}(X,1,Y)$ is $\alpha(-X^p+Y^p)$, the condition $\xi \neq 1$ yields $(X-\xi Y)\nmid (-X^p+Y^p)=(X-Y)^p$. If the two putative components through $(0:0:1)$ defined by 
\[
G(X,Y)= G_{s}(X,Y)+G_{s+1}(X,Y)\cdots
\]
and 
\[
H(X,Y)= H_{r}(X,Y)+G_{r+1}(X,Y)\cdots
\]
share the tangent line $X-\xi Y=0$ then $(X-\xi Y) \mid  G_{s}(X,Y)$ and $(X-\xi Y) \mid  H_{r}(X,Y)$ and thus 
\[
(X-\xi Y)\mid G_{s+1}(X,Y)H_{r}(X,Y)+G_{s}(X,Y)G_{r+1}(X,Y)=(X-Y)^p,
\]
a contradiction. Thus the repeated tangent lines (if any) at $(0:0:1)$ must be tangent lines of a unique (putative) component through it and thus the two components do not share any tangent line through $(0:0:1)$. As above, this shows that the multiplicity of the intersection of the two putative components through it is at most $(p-1)^2/4$, again by Lemma~\ref{le:ordinary_singular}. As before, this also shows that if the components have $r_1$ and $r_2$ tangent lines through $(0:0:1)$ (respectively), then the multiplicity of the intersection would be precisely $r_1r_2$.


Concerning the affine singular points of $\mathcal{C}_{\alpha}$, they are contained in the set of affine singular points of the curve $\frac{1}{\alpha}F_{\alpha}(X,Y)=0$ and thus they satisfy 
$$\begin{cases}
    \frac{1}{\alpha}\cdot \frac{\partial F_{\alpha}}{\partial X}= -X^{p-2}Y^p+Y^p+X^{p-2}-1=(Y^p-1)(1-X^{p-2})=0\\
    \frac{1}{\alpha}\cdot \frac{\partial  F_{\alpha}}{\partial Y}= Y^{p-2}X^p-X^p-Y^{p-2}+1=(X^p-1)(1-Y^{p-2})=0.\\
\end{cases}$$
In particular, no affine singular points belong to $XY=0$ (that can be seen easily: if $X=0$, for example, then $Y=1$, and so, $F_{\alpha}(0,1)=\alpha+2\neq 0$, a contradiction).
Also note that $F^{(1)}_{\alpha}(X,Y)$ satisfies 
$$F^{(1)}_{\alpha}(X,X)=\alpha(X^{2p-2}-X^{p}-X^{p-2}+1)=\alpha(X-1)^{p}(X^{p-2}-1)$$
and thus the only possibly singular point on the line $X=Y$ is $(1:1:1)$. 

We consider a number of cases.

\begin{enumerate}
\item $X=1$. Since 
$$\frac{1}{\alpha}F_{\alpha}(1,Y)=Y^p-Y^{p-1}+Y^p-Y+\alpha(Y^p-1)+Y^{p-1}-1+Y-1=(\alpha+2)(Y^p-1),$$
then $Y=1$, since $\alpha\neq -2$. 
\item $Y=1$. As above, this yields $X=1$, since
 $\alpha\neq -2$.
\item $X\neq1\neq Y$ and $X^{p-2}=1=Y^{p-2}$. This yields 
\begin{eqnarray*}
    \frac{1}{\alpha}F_{\alpha}(X,Y)&=& XY^2-X^{2}Y+XY^{2}-X^{2}Y+\alpha (Y^2-X^2)+Y-X+Y-X\\
    &=&(X-Y)(2XY+\alpha (X+Y)+2)=0.
\end{eqnarray*}
By the considerations above, we can assume that $X\neq Y$ and thus we are left with the condition 
$2XY+\alpha (X+Y)+2=0$. Raising it to the power $p$ and using again $X^{p-2}=1=Y^{p-2}$ one gets
$2X^2Y^2+\alpha^p (X^2+Y^2)+2=0$. Combining the two conditions it is readily seen that there are at most $4$ singular points of this type, since  $\alpha\neq -2$.
Let $(u,w)$, $u,w \notin \{0,1\}$, $u\neq w$, be a singular point. Then  
\begin{eqnarray*}
    \frac{1}{\alpha}F_{\alpha}(X+u,Y+w)&:=&u^{p-3}(w-1)^pX^2-(u-1)^pw^{p-3}Y^2+\cdots. 
\end{eqnarray*}
Thus $(u,w)$ is an ordinary double point and the intersection multiplicity of the two putative components of $\mathcal{C}_{\alpha}$  is at most $1$.
\end{enumerate}

To conclude the investigation of the singular points of $\mathcal{C}_{\alpha}$, we need to consider $(1:1:1)$. First, we compute 
\begin{eqnarray*}
    \frac{1}{\alpha}F_{\alpha}(X+1,Y+1)&=&(X+1)^{p-1}(Y^p+1)-(Y+1)^{p-1}(X^p+1)\\
    &&+XY^p+Y^p+X-YX^p-X^p-Y+\alpha(Y-X)^p\\
    &&+(Y+1)^{p-1}-(X+1)^{p-1}+Y-X\\
    &=&(X+1)^{p-1}Y^p-(Y+1)^{p-1}X^p\\
    &&+XY^p+Y^p-YX^p-X^p+\alpha(Y-X)^p\\
    &=&(\alpha+2)(Y^p-X^p)+Y^p(X^2-X^3+\cdots -X^{p-2}+X^{p-1})\\
    &&-X^p(Y^2-Y^3+\cdots -Y^{p-2}+Y^{p-1}).
\end{eqnarray*}
Thus, $\frac{1}{\alpha(Y-X)}F_{\alpha}(X+1,Y+1)$ can be written in degree increasing homogeneous components as
$$(\alpha+2)(Y-X)^{p-1} +X^2Y^2\frac{Y^{p-2}-X^{p-2}}{Y-X}+\cdots .$$
This shows that $(1:1:1)$ is a non-ordinary $(p-1)$-fold singular point. Also, two putative factors of the above polynomial  must have the following shape
$$L(X,Y):= \ell \cdot (Y-X)^{r}+L_{r+1}(X,Y)+L_{r+2}(X,Y)+\cdots$$
and 
$$M(X,Y):=m \cdot (Y-X)^{p-1-r}+M_{p-r}(X,Y)+M_{p-r+1}(X,Y)+\cdots,$$
where $L_i$ and $M_i$ are homogeneous polynomials of degree $i$ or the zero polynomial, $0<r<p-1$, $\ell m=\alpha+2\neq 0$.

These two factors correspond to two components $\mathcal{L}:L(X-1,Y-1)=0$ and $\mathcal{M}:M(X-1,Y-1)=0$ of $\mathcal{C}_{\alpha}$ passing both through $(1:1:1)$. By the properties of the intersection multiplicities, 
$$ I((0:0:1),(L(X,Y)=0)\ \cap\  (M(X,Y)=0))=I((1:1:1),\mathcal{L}\cap \mathcal{M}).$$

Since $L(X,Y)M(X,Y)=\frac{1}{\alpha(Y-X)}F_{\alpha}(X+1,Y+1)$, 

$$mL_{r+1}(X,Y) (Y-X)^{p-1-r}+\ell (Y-X)^{r}M_{p-r}(X,Y)=0$$
and 
$$\ell \cdot (Y-X)^{r}M_{p-r+1}+ L_{r+1}(X,Y)M_{p-r}(X,Y)+m \cdot (Y-X)^{p-1-r}L_{r+2}(X,Y)=X^2Y^2\frac{Y^{p-2}-X^{p-2}}{Y-X}.$$

If $r\neq (p-1)/2$ then from the first condition either \begin{align*}
(Y-X)\mid M_{p-r}(X,Y) \text{ or } (Y-X)\mid L_{r+1}(X,Y)  
\end{align*}
and thus from the second equality 
$$(Y-X)\mid X^2Y^2\frac{Y^{p-2}-X^{p-2}}{Y-X},$$
a contradiction. 
 
 Thus, if two components of $\mathcal{C}_{\alpha}$ pass through $(1:1:1)$ then such a point is a $(p-1)/2$-fold singular point for each of them and $M_{(p+1)/2}=-mL_{(p+1)/2}(X,Y)/\ell $. Also, arguing as above, the same contradiction arises if $(Y-X) \mid M_{(p+1)/2}(X,Y)L_{(p+1)/2}(X,Y)$. Thus, in the following, we assume that $(Y-X) \nmid M_{(p+1)/2}(X,Y)L_{(p+1)/2}(X,Y)$. Note that since $(1:1:1)$ is a $(p-1)/2$-fold singular point for each of the components, the degree of each of them is at least $(p-1)/2+1$.

This also tells us that the multiplicity of intersection of $\mathcal{L}$ and $\mathcal{M}$ at $(1:1:1)$ is precisely $(p^2-1)/4$. 
 In fact, by another property of the intersection multiplicities, such a number equals 
 $$I((0:0:1),(L(X,Y)=0)\ \cap\  (M(X,Y)=0))=I((0:0:1),(L(X,Y)=0)\ \cap\  (M^{\prime}(X,Y)=0)),$$
 where  $M^{\prime}(X,Y):= M(X,Y)-m L(X,Y)/\ell$.
 Now, 
\begin{eqnarray*}
M^{\prime}(X,Y)&:=&M(X,Y)-m L(X,Y)/\ell\\
&=&(M_{(p+1)/2}(X,Y)-m L_{(p+1)/2}(X,Y)/\ell)\\
&&\qquad\qquad \qquad  +(M_{(p+3)/2}(X,Y)-m L_{(p+3)/2}(X,Y)/\ell)+\cdots \\
&=&-2m L_{(p+1)/2}(X,Y)/\ell+(M_{(p+3)/2}(X,Y)-m L_{(p+3)/2}(X,Y)/\ell)+\cdots,
\end{eqnarray*}
and $$I((0:0:1),(L(X,Y)=0)\ \cap\  (M^{\prime}(X,Y)=0))=(p^2-1)/4$$
since $L(X,Y)=0$ and $M^{\prime}(X,Y)=0$ do not share any tangent line at the origin as $$GCD(L_{(p+1)/2}(X,Y), (Y-X)^{(p-1)/2})=1.$$

We are now in position to prove that $\mathcal{C}_{\alpha}$ is absolutely irreducible.

By way of contradiction, let 
\begin{align*}
\mathcal{C}^{(1)}_{\alpha}&: X^{r_1}Y^{r_2}+\cdots =0,\\
\mathcal{C}^{(2)}_{\alpha}&: X^{p-1-r_1}Y^{p-1-r_2}+\cdots =0 
\end{align*}
be two (not necessarily irreducible) components of $\mathcal{C}_{\alpha}$. They do not share any component, since otherwise the number of singular points of $\mathcal{C}_{\alpha}$ would be infinite, a contradiction.

They intersect, by B\'ezout Theorem, in precisely 
$$(r_1+r_2)(p-1-r_1+p-1-r_2)$$
points counted with multiplicity. In addition, $\mathcal{C}^{(1)}_{\alpha}$ and $\mathcal{C}^{(2)}_{\alpha}$ must intersect at singular points of $\mathcal{C}_{\alpha}$. We already proved that the only singular points of $\mathcal{C}_{\alpha}$ are $(1:0:0)$, $(0:1:0)$, and $(1:1:1)$ together with at most four other affine ordinary double points. We also showed that the multiplicity of the intersection of $\mathcal{C}^{(1)}_{\alpha}$ and $\mathcal{C}^{(2)}_{\alpha}$ at $(1:0:0)$, $(0:1:0)$ is in total 
$$r_1(p-1-r_1)+r_2(p-1-r_2).$$
If at least one between $\mathcal{C}^{(1)}_{\alpha}$ and $\mathcal{C}^{(2)}_{\alpha}$ does not pass through $(1:1:1)$ then the sum of their intersection multiplicities is
$$r_1(p-1-r_1)+r_2(p-1-r_2)+4<(r_1+r_2)(p-1-r_1+p-1-r_2),$$
a contradiction. 

Thus, $\mathcal{C}^{(1)}_{\alpha}$ and $\mathcal{C}^{(2)}_{\alpha}$ intersect at $(1:1:1)$. As we have already observed, since the only possibility is that $(1:1:1)$ is a $(p-1)/2$-fold singular point for each of the components, their degree is at least $(p-1)/2+1$. Thus, $r_1+r_2>(p-1)/2$ and $p-1-r_1+p-1-r_2>(p-1)/2$, yielding $(p-1)/2<r_1+r_2<3(p-1)/2$. 

Since the curve $\mathcal{C}_{\alpha}$ is fixed by $(X,Y)\mapsto (Y,X)$ either it has only two absolutely irreducible components (and they can be either switched or fixed by $(X,Y)\mapsto (Y,X)$) or it possesses more than two absolutely irreducible components and they can be rearranged into two larger ones fixed by $(X,Y)\mapsto (Y,X)$. 

That is to say, we can assume that either $r_1=r_2$ or $r_1=p-1-r_2$.  

If $r_1=r_2$, from $(p-1)/2<r_1+r_2<3(p-1)/2$, we have 
$(p-1)/4<r_1<3(p-1)/4$ and thus, from $p>7$,  $(p^2-1)/4+4<2r_1(p-1-r_1)$. This yields
    \begin{eqnarray*}
        r_1(p-1-r_1)+r_1(p-1-r_1)+(p^2-1)/4+4&<&2r_1(p-1-r_1) +2r_1(p-1-r_1)\\
        &=&4r_1(p-1-r_1)=\deg(\mathcal{C}^{(1)}_{\alpha})\deg(\mathcal{C}^{(2)}_{\alpha}),
    \end{eqnarray*}
a contradiction  to B\'ezout Theorem.
    
Suppose now  $r_1=p-1-r_2$. Since $r_1(p-1-r_1)\leq (p-1)^2/4$, and, from $p>7$,  $(p^2-1)/4+4<(p-1)^2/2$,
    \begin{eqnarray*}
        r_1(p-1-r_1)+r_1(p-1-r_1)+(p^2-1)/4+4&\leq &(p-1)^2/2+(p^2-1)/4+4\\
        &< &(p-1)^2=\deg(\mathcal{C}^{(1)}_{\alpha})\deg(\mathcal{C}^{(2)}_{\alpha}), 
    \end{eqnarray*}
 contradicting B\'ezout Theorem.   

The argument above shows that the curve $\mathcal{C}_{\alpha}$ is absolutely irreducible. We put together the above discussion.

\begin{thm}
Let $\alpha\in \mathbb{F}_q^*$ and $q=p^k$, $k\geq 4$, $p>7$ prime. Then the trinomial \[
f(X)=X^{q(p-1)+1}+ \alpha X^{pq}+X^{q+p-1} 
\]
is not a permutation polynomial over $\F_{q^2}$. 
\end{thm}
\begin{proof}
As already observed if $\alpha=-2$ then $g_{\alpha}(1)=0$ and thus $g_{\alpha}$ does not permute $\mu_{q+1}$. 

Suppose now that $\alpha\neq -2$. The curve $\mathcal{C}_{\alpha}$ is absolutely irreducible and so is $\mathcal{D}_{\alpha}$. Since $\mathcal{D}_{\alpha}$ is defined over $\mathbb{F}_q$ and of degree $p-1$, Theorem~\ref{Th:AubryPerret} tells us that it possesses at least 
$$ p^{k}+1-(p-2)(p-3)p^{k/2}$$
$\mathbb{F}_q$-rational points in $\mathbb{P}^2(\mathbb{F}_q)$ and at most $2(p-1)$ of them belong to the line at infinity or to $X-Y=0$. Since $k\geq 4$, $$p^{k}+1-(p-2)(p-3)p^{k/2}-2(p-1)>0.$$
Thus there exists a pair $(\overline{x},\overline{y})\in \mathbb{F}_q^2$, $\overline{x}\neq \overline{y}$, such that  $g_{\alpha}((\overline{x}+i)/(\overline{x}-i)=g_{\alpha}((\overline{y}+i)/(\overline{y}-i))$ and therefore $g_{\alpha}$ does not permute $\mu_{q+1}$. This shows that  $f(X)$ is not a permutation over $\mathbb{F}_{q^2}$.
\end{proof}

\section{Proof of the conjecture for \texorpdfstring{$k=3$}{}}

The cases $k=2,3$ require separate approaches by different algebraic number theory methods. 
Notice that $f$ can be written as
\[
f(X) = \alpha X^{pq} + \Tr(X^{q+p-1}) = (\alpha X^p + \Tr(X^{q+p-1}))^q,
\]
where $\Tr$ is the relative trace map from $\F_{q^2}$ to $\F_q$ given by $\Tr(X)=X^q+X$, $\alpha\in\F_{p^3}\setminus\F_p$ (though our proof here covers this case, as well, it is straightforward to show that if $\alpha\in\F_p$, the conjecture is true). Note that $(\alpha X^p + \Tr(X^{q+p-1}))^q$ is a permutation polynomial if and only if $\alpha X^p + \Tr(X^{q+p-1})$ is a permutation polynomial. Therefore,  in what follows, we shall consider the permutation property of the polynomial
\[
f(X)= \alpha X^p + \Tr(X^{q+p-1}),\ \alpha\in \F_q^*.
\]
Recall that $f$ is a permutation on $\F_{q^2}$ if and only if for any $g \in \F_{q^2}$, equation $f(X)=g$ has exactly one solution in $\F_{q^2}$.

Thus, for $k=3$, we consider the equation
\begin{equation}
\label{eq1}
\alpha X^p+\Tr(X^{p^3+p-1})=\alpha X^p+X^{p^3+p-1}+X^{p^4-p^3+1}=g.
\end{equation}
Raising to the $p^3$ power, we get (we use below the fact that $\alpha\in\F_{p^3},X\in\F_{p^6}$, so, $\alpha^{p^3}=\alpha,X^{p^6}=X$)
\begin{equation}
\label{eq2}
\alpha X^{p^4}+X^{p^3+p-1}+X^{p^4-p^3+1}=g^{p^3}.
\end{equation}
 Equations~\eqref{eq1} and~\eqref{eq2} imply
\[
\alpha(X^{p^4}-X^p)+g-g^{p^3}=0.
\]
We use the transformation $g\mapsto h^p,\alpha\mapsto \beta^p$, obtaining
\[
X^{p^3}-X-B=0,
\]
where
$\displaystyle B=\frac{h^{p^3}-h}{\beta}$ (we use the usual convention that $x^{-1}=x^{Q-2}$ in $\F_Q$, for some $p$-power~$Q$).

 We apply Lemma~\ref{L1} to $X^{p^3}-X-B=0$, and so $n=3$, $\ell=6$, $d=\gcd(n,\ell)=3$, $A=1$, $m=2$.  Thus, $\alpha_{m-1}=\alpha_1=1$, and therefore, either the equation has no roots or it has $p^3$ roots, which  are  of the form $x+\delta\tau$, where $\delta\in\F_{p^3}$, $\tau$ is a $(p^3-1)$-root of~$1$ (thus, we can merge them into $\lambda\in\F_{p^3}$), and for any (fixed) $c\in \F_{q^2}=\F_{p^6}$ with $\Tr(c)\neq 0$, then (using $B^{p^3}=-B$)
 \begin{align*}
 \displaystyle x&=\frac{1}{\Tr (c)} \sum_{i=0}^{1} \left(\sum_{j=0}^{i} c^{p^{3j}}\right)  B^{p^{3i}}=\frac{1}{\Tr (c)} \left(cB+\left(c+c^{p^3}\right)B^{p^3}\right)=\frac{-c^{p^3} B }{\Tr (c)}.
 \end{align*}
 We can safely take $c=1$, since $\Tr(1)=2\neq 0$, so $x=-B/2$.
 We take such a solution $X=-B/2+\lambda$, $\lambda\in\F_{p^3}$ and plug it into~\eqref{eq1}  and attempt to find two values of $\lambda$ for some fixed $h$ satisfying~\eqref{eq1}, or no value of $\lambda$ for such an $h$. 
 We therefore obtain
 \begin{align*}
& \beta^p \left(-\frac{B^p}{2}+ \lambda^p \right)
+\frac{\left(-\frac{B^{p^3}}{2}+ \lambda^{p^3}\right) \left(-\frac{B^p}{2}+ \lambda^p \right)}{\left(-\frac{B}{2}+ \lambda\right)}
+\frac{\left(-\frac{B^{p^4}}{2}+  \lambda^{p^4}\right)\left(-\frac{B }{2}+ \lambda \right)}{\left(-\frac{B^{p^3}}{2}+ \lambda^{p^3}\right)}=h^p.
 \end{align*}
 Surely, $\lambda=\frac{B}{2}$ (hence $X=0$) is a solution only for $h=0$.
 Using $B^{p^3}=-B, \lambda^{p^3}=\lambda$, we get
 \begin{equation*}
 \label{eq:lambda}
-B^2(B \beta +2    h)^p+8 \lambda B^{p+1}+4 \lambda^2 \left(B \beta +2 h\right)^p+2\lambda^p  B^2 \left(\beta -2\right)^p-8\lambda^{p+2} \left(2+ \beta\right)^p=0,
 \end{equation*}
 that is,
\begin{equation}
 \label{eq:lambda2}
\lambda^{p+2}-\frac{  B^2 \left(\beta -2\right)^p}{4\left(2+ \beta\right)^p}\lambda^p-\frac{\left(B \beta +2 h\right)^p}{2\left(2+ \beta\right)^p}\lambda^2-\frac{B^{p+1}}{\left(2+ \beta\right)^p}\lambda+\frac{B^2(B \beta +2    h)^p}{8\left(2+ \beta\right)^p}=0.
 \end{equation}
We divide by $B^{p+2}$ and using the substitution $\gamma=\frac{\lambda}{B}$, we obtain
\begin{equation}
 \label{eq:gamma}
 \gamma^{p+2}-\frac{(\beta-2)^p}{4(\beta+2)^p} \gamma^p-\frac{(\beta+\frac{2h}{B})^p}{4(\beta+2)^p} \gamma^2-\frac{1}{(\beta+2)^p} \gamma+\frac{(\beta+
 \frac{2h}{B})^p}{(\beta+2)^p}=0.
 \end{equation}
 
 We now either need to show that for some $h$, we have no solutions, or find some $h$ outside~$\F_{p^3}$ for which we have at least two values of $\gamma$  satisfying $\gamma^{p^3}=-\gamma$.
Replacing $B=\frac{h^{p^3}-h}{\beta}$, above, we get
(reverting back to $\alpha=\beta^p$ and using the notation $t:=(h^{p^4}+h^p)/(h^{p^4}-h^p)$)
\[
\gamma^{p+2}-\frac{1}{4}\gamma^p+\frac{1}{(\alpha+2)}\gamma^p -\frac{\alpha t}{4(\alpha+2)} \gamma^2-\frac{1}{(\alpha+2)} \gamma+\frac{\alpha t}{(\alpha+2)} =0,
\]
or
\begin{equation}
\label{eq:fin3}
\gamma^{p+2}- \frac{1-4\mu}{4}  \gamma^p -  \frac{(1-2\mu) t}{4}\, \gamma^2-\mu\, \gamma+ (1-2\mu) t =0,
\end{equation}
where $\mu=\dfrac{1}{\alpha+2}$ (so, $\mu\alpha=1-2\mu$).
We just need to find some $h\notin\F_{p^3}$ such that the previous equation has no solutions or it has more than one. 

One wonders if the simplest approach would work, that is, find some $h$ such that $t=0$. Surely,
if $h^{p^3}=-h$, then $t=0$, and Equation~\eqref{eq:fin3} becomes
\begin{equation}
    \label{eq:fin3_1}
\gamma^{p+1}+\frac{1-4\mu}{4\mu} \gamma^p-\mu\gamma=0.
\end{equation}
Surely, $\gamma=0$ is a root. Now, assuming $\gamma\neq 0$, we remove one copy of $\gamma$ and divide the above equation of $\gamma^{p+1}$. Relabeling $\frac{1}{\gamma}\mapsto \gamma$, we obtain 
\[
\gamma^{p+1}+\frac{1-4\mu}{4\mu}\gamma^2-\frac1\mu=0.
\]
We now let $\zeta\in\F_{p^3}$ such that $\zeta^{p^2+p+1}=-1$ (there are many such examples, like $\zeta=-1$, or $\zeta=6g^5+2g^4+2g^3+10g^2+5g+5$, where $g$ is a primitive root in $\F_{p^6}$, $p=11$, defined via the primitive polynomial $x^6+3x^4+4x^3+6x^2+7x+2$). We look for $\gamma$ such that $\gamma^p=\zeta \gamma$ (and so, our choice of $\zeta$ renders $\gamma^{p^3}=-\gamma$, which is needed).

The above displayed equation now becomes
\[
\gamma^2\left(\zeta+\frac{1-4\mu}{4\mu}\right)=\frac1\mu,
\]
that is $\gamma^2=\frac4{4(\zeta-1)\mu+1}$. Assuming that $4(\zeta-1)\mu+1$ is not a perfect square in $\F_{p^3}$, then there are values of $\gamma\neq 0$ satisfying~\eqref{eq:fin3_1}.
In what follows, we use $\eta(a)$ to denote the quadratic character of~$a$, that is, $\eta(a)=0$, if $0$, ands $\eta(a)=1$, if $a\neq 0$ is a square, respectively, $\eta(a)=-1$, if $a$ is not a square in $\F_{q}$.

We will argue now that given $\mu=\frac1{\alpha+2}\in\F_{p^3}\setminus\F_p$, we can always find $\zeta$ with $\eta(\zeta)=-1$ such that $\eta(4(\zeta-1)\mu+1)=-1$. Thus, we need to show that the below sum is nonzero, namely,
\begin{align*}
&\frac14 \sum_{\zeta\in\F_q} \Big (1-\eta(\zeta) \Big) {\Big(}1-\eta(4(\zeta-1)\mu+1)\Big)=\\
&\quad \frac14\left(q-\sum_{\zeta\in\F_q}\eta\Big(4(\zeta-1)\mu+1\Big)+\sum_{\zeta\in\F_q}\eta\Big({(4\left(\zeta-1)\mu+1\right)\zeta}\Big) \right)= \\
&\quad\frac14\left(q+\sum_{\zeta\in\F_q}\eta\Big({\left(4(\zeta-1)\mu+1\right)\zeta}\Big) \right),
\end{align*}
where we used the fact that $\zeta\mapsto 4(\zeta-1)\mu+1$ is linear and consequently, in the middle two sums, half the terms are $+1$ and half are $-1$. For the last expression above to vanish, we need $\eta\Big({(4(\zeta-1)\mu+1)\zeta}\Big)=-1$, regardless of $\zeta$. We will now use Weil's Theorem (see~\cite[Theorem 5.41]{LN97}), which states that for a character $\chi$ of order $s$ of $\F_q$ and a polynomial $f$ of degree $d$ over $\F_q$, which is not of the form $c(h(x))^\ell$, $c\in\F_q$, $h\in\F_q[x], \ell>1$, we have that $\left|\sum_{x\in\F_q} \chi(f(x))\right|\leq (d-1)\sqrt{q}$. Observing that ${(4(Z-1)\mu+1)Z}$ is not the square of a linear polynomial in $Z$, and taking $\chi=\eta$, the quadratic character (of order $2$), the  Weil's bound renders $\left|\sum_{x\in\F_q} \eta\Big({(4(\zeta-1)\mu+1)\zeta}\Big)\right|\leq \sqrt{q}$, and the above displayed equation cannot vanish.

What we thus showed is that regardless of $\mu=\frac1{\alpha+2}\in\F_q^*$, we can find some $\zeta\in\F_q$ with $\eta(\zeta)=-1$ such that $\gamma=\pm \sqrt{\frac4{4(\zeta-1)\mu+1}}$ are roots of Equation~\eqref{eq:fin3_1}, assuming that $\gamma^p=\zeta \gamma$ (so, $\gamma^{2p}=\zeta^2\gamma^2$).

We now look more closely at the condition $\gamma^{2p}=\zeta^2\gamma^2$. For a fixed $\zeta$ with $\zeta^{p^2+p+1}=-1$ (hence $\zeta\in\F_q$), the previous condition renders the equation in $\mu$,
\begin{equation}
\label{eq:fin3_2}
\mu^p-\frac{\zeta-1}{\zeta^2(\zeta^p-1)}\mu+\frac{\zeta^2-1}{4\zeta^2(\zeta^p-1)}=0.  
\end{equation}
We apply again Lemma~\ref{L1} with $n=1, d=1, m=3$, $A=\dfrac{\zeta-1}{\zeta^2(\zeta^p-1)}$, $B=-\dfrac{\zeta^2-1}{4\zeta^2(\zeta^p-1)}$. Thus,
\begin{align*}
\alpha_2&=A^{p^2+p+1}=\zeta^{-2-2p-2p^2}=1,\\
\beta_2&=\sum_{i=0}^2 s^{s_i}b^{p^i}=a^{s_0}b+a^{s_1} b^p+a^{s_2} b^{p^2}=\frac{\zeta^{-2p^2-2p-2}(\zeta^{2p^2+2p+2}-1)}{4(\zeta-1)}=0,
\end{align*}
since $s_0=p^2+p,s_1=p^2,s_2-0$. Consequently, Equation~\eqref{eq:fin3_2} has exactly $p$ roots, regardless of the value of $\zeta$ satisfying $\zeta^{p^2+p+1}=-1$.

We next get a lower bound for the number of $\mu$ satisfying all these conditions.
Now, if Equation~\eqref{eq:fin3_2} holds for two different $\zeta_1,\zeta_2$ (with the known conditions $\zeta_i^{p^2+p+1}=-1$, $i=1,2$), by subtracting the corresponding Equations~\eqref{eq:fin3_2}, we obtain a linear equation in $\mu$, and hence there exists at most one such $\mu$ that may satisfy~\eqref{eq:fin3_2} for two different $\zeta$'s. Therefore, the set of all $\mu$ satisfying the imposed conditions has cardinality at least $p+(p-1)+\cdots+2+1=\dfrac{p^2+p}{2}$. This is surely a rough estimate, and a computation reveals that. For example, if $p=11$, the set of $\mu$ satisfying all the imposed conditions is of cardinality $522$, compared with the cardinality of $\F_q^*$, that is, $1330$.

We collect the prior major observations in the next theorem.
\begin{thm}
Let $q=p^3$, where $p$ is prime. For any $-2\neq \alpha\in\F_q$, and $\zeta\in\F_q$ with $\zeta^{p^2+p+1}=-1$ such that $\eta(z)=-1$ ($\eta$ is the quadratic character on $\F_q$), and $z:=1+\frac{4(\zeta-1)}{\alpha+2}$ satisfying $\zeta^p=\zeta^2 \zeta$,   the trinomial 
\[
f(X)=X^{q(p-1)+1}+\alpha X^{pq}+X^{q+p-1}
\]
is not a permutation polynomial over $\F_{q^2}$.
\end{thm}

We included the prior result since one might wonder if $t=0$ would be sufficient to show the conjecture for $k=3$, but as one sees above, while it covers some ground, it cannot cover the entire set of $\alpha$'s.

However, below, for every $\alpha$ (hence $\mu$), we shall find some value of $h$ for which the prior equation has no solutions $\gamma$ (with $\gamma^{p^3}=-\gamma$).
We first show that there exists $h\neq 0$ such that
$T:=(1-2\mu)t$ satisfies $T^p=-T$.
For that, we will show the existence of $h\neq 0$ such that 
\[
\dfrac{h^{p^3}+h}{h^{p^3}-h}=\left(\frac{T}{1-2\mu} \right)^{1/p},\text{ that is, } h^{p^3}+\frac{1+\ell}{1-\ell}\, h=0,
 \]
where $\ell=\left(\frac{T}{1-2\mu} \right)^{1/p}$ (observe that $\ell^{p^3}+\ell=0$). The prior displayed equation has a nonzero root if we can show that the  linearized polynomial  $X^{p^3}+a X$ is not a permutation (since then the dimension of the kernel is greater than~$1$, and a nonzero root exists). 

We know~\textup{\cite{ZWW20}} that a linearized polynomial of the form $L(x)=x^{p^r}+a x\in\F_{p^n}$ is a permutation polynomial if and only if the relative norm $N_{\F_{p^n}/\F_{p^d}}(a)\neq 1$, that is, $(-1)^{n/d} a^{(p^n-1)/(p^d-1)}\neq 1$, where $d=\gcd(n,r)$. In our case, $r=3$, $n=6$, and  $a=\frac{1+\ell}{1-\ell}$. 
But we quickly see that since $\ell^{p^3}+\ell=0$, the norm is~$1$ and consequently the mentioned binomial is not a permutation (since the corresponding polynomial is linearized, then the kernel has dimension at least one, and so, we have at least $p$ roots). Thus, regardless of $\alpha$, such an $h$ exists (in fact, exactly $p^3$ roots, all given by $a\,h_0$, where  $a\in\F_{p^3}$ and $h_0\neq 0$) that will render $T$ with $T^p=-T$ (note that $\omega:=T^2\in\F_p$). 

We will next show that the following equation
\begin{equation}
\label{eq:fin3_3}
\gamma^{p+2}- \frac{1-4\mu}{4}  \gamma^p -  \frac{T}{4}\, \gamma^2-\mu\, \gamma+ T =0
\end{equation}
has no solution~$\gamma$ in $\F_{p^6}$ satisfying $\gamma^{p^3}=-\gamma$, or it has more than one.

Observe that since $\gamma^{p^3}=-\gamma\neq 0$, we have that $\gamma^p=\zeta\gamma$, where $\zeta=\gamma^{p-1}$ and $\zeta^{p^2+p+1}=-1$.  We now fix some $\zeta\in\F_{q^2}$ with $\zeta^{p^2+p+1}=-1$ and look for $\gamma^p=\zeta\gamma$. We will show that regardless of what $\zeta$ is, there exists no $\gamma$ or there are more than one satisfying~\eqref{eq:fin3_3}. We thus consider
\[
g(\gamma):=\zeta \gamma^3-\frac{1-4\mu}{4}  \zeta \gamma -  \frac{T}{4}\, \gamma^2-\mu\, \gamma+ T=0.
\] 
If the above equation  has no solution $\gamma$ with $\gamma^p=\zeta\gamma$, we are done. If it has one solution in $\F_{q^2}$,  we will actually  show that it has more than one.  
Since the characteristic $p\ne 2,3$, we can easily solve the equation (we used SageMath, here) using classical formulas (we use $\gamma_i$, $i=1,2,3$, to denote the roots in some extension of $\F_{q^2}$), and obtain (for easy writing, and also because it will be useful later, we use  $D:=\left({\frac{T (\omega -18 \zeta (4 \mu  (\zeta-1)+47 \zeta))+6 \zeta \sqrt{-48 \omega ^2+3 \omega  \left(6623 \zeta^2-16 \mu ^2  (\zeta-1)^2+1160 \mu  (\zeta-1) \zeta\right)-48 \zeta (4 \mu -4 \mu  \zeta+\zeta)^3}}{\zeta^3}}\right)^{1/3}$)
\allowdisplaybreaks
\begin{equation}
\begin{split}
\label{eq:gammas1}
\gamma_1&=\frac{1}{24} \left(\frac{\left(1-\sqrt{-3}\right) (12 \zeta (4 \mu  (\zeta-1)-\zeta)-\omega )}{D \zeta^2}- \left(1+\sqrt{-3}\right) D+\frac{2 T}{\zeta}\right),\\
   \gamma_2&=\frac{1}{24} \left(\frac{\left(1+\sqrt{-3}\right)
   (12 \zeta (4 \mu  (\zeta-1)-\zeta)-\omega )}{D \zeta^2}-\left(1-\sqrt{-3}\right) D+\frac{2 T}{\zeta}\right),\\
   \gamma_3&= \frac{\zeta
   \left(\zeta \left(D^2-48 \mu +12\right)+D T+48 \mu \right)+\omega }{12 D \zeta^2}.
   \end{split}
   \end{equation}
We note that $1\pm \sqrt{-3}\in\F_{q^2}$ (since the splitting field of $x^2+3$ is either $\F_p$, or $\F_{p^2}$, which is a subfield of $\F_{q^2}$).
We assume that  at least one of these roots in~\eqref{eq:gammas1}  are in $\F_{q^2}$. Thus, at least one of $\gamma_i^{p^3}+\gamma_i=0$ holds. Via the quadratic reciprocity, we observe that if $a=\pm\sqrt{-3}$, then $a^p=a$, when $p\equiv 1\pmod 6$, and $a^p=-a$, when $p\equiv 5\pmod 6$; since the computations are similar, sure, with different expressions that are common to all factorizations, without loss of generality, we consider only the case of $a^p=a$, so $p\equiv 1\pmod 6$. We fully simplify the expressions $\gamma_i^{p^3}+\gamma_i$,   and we infer that at least one of the following  must vanish, namely
\begin{equation}
\begin{split}
\label{eq:gammas_p3}
& \left(D^{p^3}+D\right)  \left(-1+\sqrt{-3}\right) \left(\zeta^2 \left(\theta D^{p^3+1}-48 \mu +12\right)  +48 \mu  \zeta+\omega \right),\\
 &  \left(D^{p^3}+D\right) \left(-1-\sqrt{-3}\right) \left(\zeta^2 \left(\frac{1}{\theta} D^{p^3+1}-48 \mu +12\right)+48 \mu  \zeta+\omega \right),\\
& \left(D^{p^3}+D\right) \left(\zeta^2 \left(D^{p^3+1}-48 \mu +12\right)  +48 \mu  \zeta+\omega \right),
 \end{split}
\end{equation}
where $\theta=\frac{\left(1+\sqrt{-3}\right)}{\left(1-\sqrt{-3}\right)}$ (observe that $\theta^2+\theta+1=0$).
As a side note, from the above relations, we  also see that the roots are distinct.
If $D^{p^3}+D=0$, then all $\gamma_i$ satisfy $\gamma_i^{p^3}+\gamma_i=0$, and so, if one root is in $\F_{p^6}$, all roots are there. If $D^{p^3}+D\neq 0$, then the second parenthesis in one of the prior expressions must vanish. Therefore, regardless of which one does vanish, $D^{p^3+1}$ must be in $\F_{p^3}$ (note that  $\gcd(p^3+1,p^3-1)=2$), so $D\in\F_{p^6}$. Thus, there exist $a,b\in\F_{p^3}$ such that $D^2=aD+b$, $ab\neq 0$. Without loss of generality, we assume that $\gamma_1\in\F_{p^6}$.  
Now, from the expression of $D$, we infer that
\begin{align*}
D^3&=D(aD+b)=(a^2+b)D+ab=a_1D+b_1=T c_1+\sqrt{d_1},\\
&a_1=(a^2+b),b_1=ab\neq 0,c_1=\frac{ (\omega -18 \zeta (4 \mu  (\zeta-1)+47 \zeta))}{\zeta^3},\\
&d_1=\frac{36(-48 \omega ^2+3 \omega  \left(6623 \zeta^2-16 \mu ^2  (\zeta-1)^2+1160 \mu  (\zeta-1) \zeta\right)-48 \zeta (4 \mu -4 \mu  \zeta+\zeta)^3)}{\zeta^2},
\end{align*}
with all parameters in $\F_{p^3}$, and we can write 
\[
a_1^2 b + b_1^2 + c_1^2\omega+  a_1(a a_1 + 2 b_1 ) D  - 2 c_1 (b_1 + a_1 D) T =d_1,
\]
or,
\begin{align*}
D a_1\left(a a_1+2  b_1-2 c_1 T\right)+a_1^2 b+b_1^2+c_1^2 \omega-2 b_1 c_1 T = d_1.
\end{align*}
If $a_1=0$,  then $D^3=ab\in\F_{p^3}$. Since $3\,|\,p^3-1$,  for $p\equiv 1\pmod 6$, and so, $D^{p^3-1}=(D^3)^{(p^3-1)/2}=1$, therefore $D\in\F_{p^3}$.
Now, we let $a_1\neq 0$. If $c_1=0$, then $a_1 D+b_1=\sqrt{d_1}$, which by squaring (and replacing $D^2=aD+b$), we get $b (a^4 + 3 a^2 b + b^2) + a (a^2 + b) (a^2 + 3 b) D=d_1$, so   
 $D\in\F_{p^3}$, unless $a^2+3b=0$, but then $b(a^4+3a^2b+b^2)=0$ would imply $b=0=a$, a contradiction to $a_1\neq 0$.

 We can now assume that $a_1c_1\neq 0$.
 Thus, the coefficient of $D$ in the prior displayed equation is not zero (since $T\notin\F_{p^3}$), and so $D=\frac{u_1+v_1T}{u_2+v_2T}=\frac{(u_1+v_1T)(u_2-v_2T)}{u_2^2-v_2^2 \omega}=\frac{u_1u_2-v_1v_2\omega+(u_2v_1-u_1v_2)T}{u_2^2-v_2^2 \omega}=A+BT$, for  $A=\frac{u_1u_2-v_1v_2\omega}{u_2^2-v_2^2 \omega}$, $B=\frac{(u_2v_1-u_1v_2)}{u_2^2-v_2^2 \omega}$, and $u_1=d_1-a_1^2b-b_1^2-c_1^2\omega$, $v_1=2b_1c_1$, $u_2=a_1(a a_1+2b_1)$, $v_2=-2a_1c_1$.

 Plugging this back into $a_1D+b_1=Tc_1+\sqrt{d_1}$, expanding, separating the radical and squaring, to get rid of it, we obtain the following equation involving $T$ and coefficients in $\F_{p^3}$, 
 {\small
 \begin{align*}
\frac{2 c_1  \left(b^3+c_1^2 \omega
 -d_1\right)}{\left(a^2+b\right) \left((a^3+3ab)^2-4 c_1^2
   \omega \right)}T+\frac{a \left(a^6 b+6 a^4 b^2+a^2 \left(10 b^3+c_1^2 \omega-d_1\right)+b \left(3 b^3-c_1^2 \omega -3 d_1\right)\right)}{\left(a^2+b\right) \left((a^3+3ab)^2-4 c_1^2 \omega \right)}=0.
 \end{align*}
 }
Note that the denominator is not zero, since we assumed that $a_1\ne 0$ and if the second parenthesis were zero, then $T^2=\omega$ would be a square in $\F_{p^3}$ and that is also not possible. 
 
 Since $1,T$ are independent over $\F_{p^3}$,  their coefficients must be zero. We further simplify and find that the only values of $c_1,\sqrt{d_1}$ are in fact
 \begin{align*}
  c_1&=0, \sqrt{d_1}= \frac{d_1-b^3}{a \left(a^2+3
   b\right)},\text{ or}\\
  c_1&= \pm\sqrt{\frac{a^6+6 a^4 b+9 a^2 b^2+2 b^3-2
   d_1}{ 2\omega }},\sqrt{d_1}= \frac{1}{2} a \left(a^2+3   b\right).
 \end{align*}
Going diligently through all these simple cases, we infer that $D$, or $D^2\in\F_{p^3}$ (this last one can happen when $c_1=0$, but then $D^{p^3}=\pm D$ (since $\gcd(p^3+1,p^3-1)=2$), which under the condition that $D^{p^3}+ D\neq 0$, renders  also that $D\in\F_{p^3}$). Regardless of the case, we always get that $D\in\F_{p^3}$, 
but then all of $\gamma_i\in\F_{p^3}$, and so, we cannot have any of them satisfy $\gamma_i^{p^3}+\gamma_i=0$, unless that particular $\gamma_i=0$, but that is impossible, since~$T\neq 0$.

 We put together the previous discussions in the next theorem  (recall, and easy to show, that if $\alpha\in\{0,-2\}$, the polynomial is easily shown to be a non-permutation).
 \begin{thm}
Let $q=p^3$, where $p$ is a prime. Then, for $\alpha \in \F_q$,  the trinomial
\[
f(X)=X^{q(p-1)+1}+ \alpha X^{pq}+X^{q+p-1} 
\]
is not a permutation polynomial over $\F_{q^2}$.
\end{thm} 

\section{Proof of the conjecture for \texorpdfstring{$k=2,\alpha= -1$}{}}

We start with showing the permutation property of the polynomial when $k=2,\alpha=-1$.
\begin{thm}
Let $q=p^2$, where $p$ is a prime. Then the polynomial
\[
f(X)=X^{q(p-1)}-X^{pq}+X^{q+p-1}
\]
is a permutation polynomial on $\F_{q^2}$.
\end{thm}
\begin{proof}
As for the prior case of $k=3$, we just need to consider the equation
\begin{equation}
\label{eq1_2}
\alpha X^p+\Tr(X^{p^2+p-1})=\alpha X^p+X^{p^2+p-1}+X^{p^3-p^2+1}=g=h^p,
\end{equation}
where $g\in\F_{p^4}$.
Raising to the $p^2$ power, we get (we use  that $\alpha\in\F_{p^2},X\in\F_{p^4}$, so, $\alpha^{p^2}=\alpha,X^{p^4}=X$)
\begin{equation}
\label{eq2_2}
\alpha X^{p^3}+X^{p^2+p-1}+X^{p^3-p^2+1}=h^{p^3}.
\end{equation}
Combining Equations~\eqref{eq1_2} and~\eqref{eq2_2}, we get
\[
\alpha(X^{p^3}-X^p)+h^p-h^{p^3}=0.
\]
Of course, if $\alpha=-1$,   we get the equation   $(X+h)^{p^2}=(X+h)$, so 
  $X+h\in\F_{p^2}$.
   We need to show that the original equation has a unique root for any $h$, to infer the necessity of the claim of the conjecture (if $k=2,\alpha=-1$). 
   We thus write $X=u-h$, $h=g^{1/p}$, $u\in\F_{p^2}$, and replace into~\eqref{eq1_2}. 
   If $h\in\F_{p^2}$, then $X\in\F_{p^2}$ and so, the original Equation~\eqref{eq1_2} simplifies to $-u^p+u^p+u^p=0$ (since $u^{p^2-1}=1$), and so, $u=0$ is the unique solution. Let $h\in\F_{p^4}\setminus \F_{p^2}$. 
Note that if $0\ne u$ is a root of~\eqref{eq1_2}, then $u^p=\zeta u$, where $\zeta^{p+1}=1$. To show that the root of~\eqref{eq1_2} is unique, we fix $\zeta\in\F_{p^2}$ with $\zeta^{p+1}=1$ and show that there exists only one $u$ vanishing~\eqref{eq1_2} such that $u^p=\zeta u$. That would show the permutation property of our polynomial.

When $u^p=\zeta u$, Equation~\eqref{eq1_2} transforms into
\begin{equation}
\begin{split}
\label{eq:u0}
&h^{p^3+2}+h^{2 p^2+p}+u \left(\zeta  \left(-h^{2 p^2}+h^{p^2+1}-h^2\right)-2 \left(h^{p^3+1}+h^{p^2+p}\right)\right)\\
&\qquad\qquad \qquad\qquad  +u^2 \left(\zeta  \left(h^{p^2}+h\right)+h^p+h^{p^3}\right)-u^3 \zeta =0.
\end{split}
\end{equation}
Raising to the $p$-power, we obtain
\begin{align*}
&h^{2  p^3+p^2}+h^{2 p+1}+u \left(-h^{2 p^3}+h^{p^3+p}+\zeta  \left(-2 h^{p^3+p^2}-2 h^{p+1}\right)-h^{2 p}\right)\\
&\qquad\qquad \qquad\qquad +u^2 \left(\zeta ^2 \left(h^{p^2}+h\right)+\zeta 
   \left(h^{p^3}+h^p\right)\right)-u^3 \zeta ^2=0.
\end{align*}
Multiplying the second equation by $\zeta$ and subtracting it from the first 
we obtain
\begin{align*}
&\zeta h^{p^3+2}+\zeta h^{2 p^2+p}-h^{2 p^3+p^2}-h^{2 p+1} +u \left(-2 \zeta h^{p^3+1}+h^{2
   p^3}-h^{p^3+p}\right.\\
   &\left. -\zeta^2 h^{2 p^2} +\zeta^2 h^{p^2+1}-2 \zeta h^{p^2+p}+2 \zeta h^{p^3+p^2}+2 \zeta h^{p+1}+h^{2
   p}-h^2 \zeta^2\right)=0,
\end{align*}
and if the coefficient of $u$ is nonzero, we get the solution
{\small
\begin{equation}
\label{eq:u}
u=\frac{-\zeta h^{p^3+2}-\zeta h^{2 p^2+p}+h^{2 p^3+p^2}+h^{2 p+1}}{-2 \zeta h^{p^3+1}+h^{2
   p^3}-h^{p^3+p}-\zeta^2 h^{2 p^2}+\zeta^2 h^{p^2+1}-2 \zeta h^{p^2+p}+2 \zeta h^{p^3+p^2}+2 \zeta h^{p+1}+h^{2
   p}-h^2 \zeta^2},
\end{equation}
}
 hence uniqueness of the solution of~\eqref{eq1_2} in $\F_{p^2}$. If the coefficient of $u$ is zero, then we must  have $\zeta=\frac{h^{2 p^3+p^2}+h^{2 p+1}}{h^{p^3+2}+h^{2 p^2+p}}\neq 0$ (it is easy to check that this $\zeta$ satisfies the required $\zeta^{p+1}=1$, if nonzero). 

 Using the above expression of $\zeta$ into the conditional expression $u^p=\zeta u$, we infer that  $u^{p-1}=\dfrac{h^{2 p^3+p^2}+h^{2 p+1}}{h^{p^3+2}+h^{2 p^2+p}}=\dfrac{B^p}{B}=B^{p-1}$, where $B=h^{p^3+2}+h^{2 p^2+p}=H+H^{p^2}$, for $H=h^{2 p^2+p}$. Thus, $u=B \epsilon$, where $\epsilon\in\F_p$. We need to find the values of $\epsilon$ such that this $u$ satisfies~\eqref{eq1_2}. Observe that $B=H+H^{p^2}$, for $H=h^{2 p^2+p}$. Plugging $u=B \epsilon$ back into~\eqref{eq1_2} and simplifying, we obtain that the only possible values is $\epsilon=0$, so $u=0$, but then $h^{2  p^3+p^2}+h^{2 p+1}=0$, from~\eqref{eq:u0}, and that is impossible since $\zeta\neq 0$. 
\end{proof}

\section{\bf The proof of the conjecture for \texorpdfstring{$k=2$, $\alpha\neq -1$}{}}

   We now assume $\alpha\neq -1$ and show that the polynomial is not a permutation, when $k=2$.
We use the transformation $g\mapsto h^p$, $\alpha\mapsto \beta^p$, obtaining
\[
X^{p^2}-X-B=0,\text{ where } B=\frac{h^{p^2}-h}{\beta}.
\]

 Applying again Lemma~\ref{L1} to $X^{p^2}-X-B=0$, we infer that the solutions must be of the form $X=-B/2+\lambda$, $\lambda\in\F_{p^2}$,  and plugging this into~\eqref{eq1_2}, we   attempt to find at least two values of $\lambda$ for some fixed $h$ satisfying~\eqref{eq1_2}, or no value of $\lambda$ for such an $h$. 

 We therefore obtain
 \begin{align*}
& \beta^p \left(-\frac{B^p}{2}+ \lambda^p \right)
+\frac{\left(-\frac{B^{p^2}}{2}+ \lambda^{p^2}\right) \left(-\frac{B^p}{2}+ \lambda^p \right)}{\left(-\frac{B}{2}+ \lambda\right)}
+\frac{\left(-\frac{B^{p^3}}{2}+  \lambda^{p^3}\right)\left(-\frac{B }{2}+ \lambda \right)}{\left(-\frac{B^{p^2}}{2}+ \lambda^{p^2}\right)}=h^p.
 \end{align*}
 Note that $\lambda=\frac{B}{2}$ (hence $X=0$) is a solution only for $h=0$.
 Using $B^{p^2}=-B, \lambda^{p^2}=\lambda$, we get, after simplifying, 
and using the substitution $\gamma=\frac{\lambda}{B}$,  
\begin{equation}
 \label{eq:gamma_2}
 \gamma^{p+2}-\frac{(\beta-2)^p}{4(\beta+2)^p} \gamma^p-\frac{(\beta+\frac{2h}{B})^p}{4(\beta+2)^p} \gamma^2-\frac{1}{(\beta+2)^p} \gamma+\frac{(\beta+
 \frac{2h}{B})^p}{(\beta+2)^p}=0.
 \end{equation}
Replacing $B=\frac{h^{p^2}-h}{\beta}$, and using the notation $t:=(h^{p^3}+h^p)/(h^{p^3}-h^p)$ (note that $t^{p^2}=-t$), we get  
\begin{equation}
\label{eq:fin3_2_2}
\gamma^{p+2}- \frac{1-4\mu}{4}  \gamma^p -  \frac{(1-2\mu) t}{4}\, \gamma^2-\mu\, \gamma+ (1-2\mu) t =0.
\end{equation}
The same argument we used for $k=3$ shows that there exists $h$ such that $T=(1-2\mu)t$ satisfies $T^p=-T\neq 0$ (here, we use $\alpha\neq -1$), and the equation becomes
\begin{equation}
\label{eq:fin_k2_1}
\gamma^{p+2}- \frac{1-4\mu}{4}  \gamma^p -  \frac{T}{4}\, \gamma^2-\mu\, \gamma+ T =0.
\end{equation}
where $\mu=\dfrac{1}{\alpha+2}$,
 which is the same equation as for the $k=3$ case, sure, with the parameters satisfying other conditions. Note that $T\in\F_{p^2}$.

We will now show that for some $h$ outside~$\F_{p^2}$, we have  at least two values of $\gamma$  satisfying $\gamma^{p^2}=-\gamma$.

 As for the prior case of $k=3$, we can assume that $\gamma$  satisfies $\gamma^p=\zeta\gamma$, where $\zeta^{p+1}=-1$ (so that $\gamma^{p^2}=-\gamma$).
We now fix $\zeta$ with $\zeta^{p+1}=-1$ and we write Equation~\eqref{eq:fin_k2_1} as
\begin{equation}
\label{eq:fin_k2_2}
\zeta \gamma^{3}-  \frac{T}{4}\, \gamma^2- \frac{1-4\mu}{4}  \zeta \gamma  -\mu\, \gamma+ T =0,
\end{equation}
whose solutions are 
of the form (below, we let $\theta=\frac{\left(1+\sqrt{-3}\right)}{\left(1-\sqrt{-3}\right)}\in\F_{p^3}$; note that $\theta^2+\theta+1=0$),
\begin{equation*}
\begin{split}
\gamma_1&= \frac{\zeta \left(-D^2 \theta  \zeta+D (\theta +1) T+48 \mu  (\zeta-1)-12 \zeta\right)-\omega }{12 D (\theta
   +1) \zeta^2},\\
   \gamma_2&=  \frac{\zeta \left(-D^2 \frac1\theta\zeta+D (\frac1\theta +1) T+48    \mu  (\zeta-1)-12   \zeta\right)- 
   \omega }{12 D (\frac1\theta +1) \zeta^2},\\
   \gamma_3&= \frac{\zeta \left(\zeta \left(D^2-48 \mu +12\right)+D T+48 \mu \right)+\omega }{12 D \zeta^2}.
\end{split}
\end{equation*}
We rewrite the above expressions as
\begin{equation}
\begin{split}
\label{eq:gammas_k2}
12\gamma_1&= D \theta^2-\frac{\theta  \left(-\omega +48 \mu  \zeta^2-12 \zeta^2-48 \mu  \zeta\right)}{D \zeta^2}+\frac{T}{\zeta}=\theta^2 D+\theta a_2 \frac{1}{D}+a_3,\\
12\gamma_2&= D \theta  -\frac{\theta ^2 \left(-\omega +48 \mu  \zeta^2-12 \zeta^2-48 \mu 
   \zeta\right)}{D \zeta^2}+\frac{T}{\zeta}=\theta D+ \theta^2 a_2 \frac{1}{D}+a_3,\\
12\gamma_3&=D-\frac{-\omega +48 \mu \zeta^2-12 \zeta^2-48 \mu \zeta}{D \zeta^2}+\frac{T}{ \zeta}=D+a_2\frac1{D}+a_3,
\end{split}
\end{equation}
for the obvious $a_2,a_3\in\F_{p^2}$,  where  
\begin{align*}
 D^3&=T \zeta^{-3} (\omega -18 \zeta (4 \mu  (\zeta-1)+47 \zeta))\\
 &\qquad+6 \zeta^{-2} \sqrt{-48 \omega ^2+3 \omega  \left(6623 \zeta^2-16 \mu ^2  (\zeta-1)^2+1160 \mu  (\zeta-1) \zeta\right)-48 \zeta (4 \mu -4 \mu  \zeta+\zeta)^3}\\
&=Tc_1+\sqrt{d_1},
\end{align*}
for the obvious $c_1,d_1\in\F_{p^2}$. Note that $0\neq D^3\in\F_{p^4}$ (since we assume that at least one root exists).
Also, at least two of the roots are distinct, as one can see from the following argument.
If $\gamma_1=\gamma_2$ (similar arguments will work if other two roots are equal), then $D=\frac{a_2}{D}$, and so, $\gamma_1=\gamma_2=(\theta^2+\theta)D+a_3=-D+a_3\neq 2D+a_3=\gamma_3$ (unless $D=0$, but that is easily assessed to be impossible). Thus, if one of the roots is in $\F_{p^4}$, then $D\in\F_{p^4}$, and the other roots will be in~$\F_{p^4}$, as well.   
 
Without loss of generality (the computations are absolutely similar), we assume that $\gamma_1\in\F_{p^4}$. 
 Using the fact that  $(12\gamma_1)^{p^4}=(12\gamma_1)$, we get
\begin{align*}
&\theta \left(D^{p^4}-D\right)\left(\theta-a_2\frac{1}{D^{p^4+1}} \right)=0.
\end{align*}
If $D^{p^4}=D$, then if our equation has a root in $\F_{p^4}$ for that given $\zeta$, then it must have all roots in $\F_{p^4}$ (since now $D\in\F_{p^4}$). If $D^{p^4}\neq D$, then $D^{{p^4}+1}=\frac{a_2}{\theta}\in\F_{p^2}$. Because $3\,|\,{p^2}-1$, for $p>3$, and $D^3\in\F_{p^4}$, then $D^{{p^4}+1}=D^2 D^{{p^4}-1}=D^2 (D^3)^{\frac{{p^4}-1}{3}}\in\F_{{p^4}}$, and therefore, $D^2\in\F_{p^4}$.  Since $D^3\in \mathbb{F}_{p^4}$ we conclude $D\in \mathbb{F}_{p^4}$, yet again. But then from $12\gamma_1=\frac{\theta^2 D^2+\theta a_2}{D}+a_3$ implies that $\frac{1}{D}$ (and therefore, $D$) must be in $\F_{p^4}$, unless $\theta^2 D^2+\theta a_2=0$, but then $D^2\in\F_{p^2}$, rendering $D\in\F_{p^4}$, yet again.

 We summarize our discussion in the next theorem.
  \begin{thm}
Let $q=p^2$, where $p$ is a prime. Then, for $-1\neq \alpha \in \F_q$,  the trinomial
\[
f(X)=X^{q(p-1)+1}+ \alpha X^{pq}+X^{q+p-1} 
\]
is not a permutation polynomial over $\F_{q^2}$.
\end{thm}

\section*{Acknowledgements}
The third-named author (PS) would like to thank  the first-named author (DB) for the invitation at the Dipartimento di Matematica e Informatica  at Universit\`a degli Studi di Perugia, Italy, and the great working conditions while this paper was being written. The first-named author (DB) thanks the Italian National Group for Algebraic and Geometric Structures and their Applications (GNSAGA—INdAM) which supported the research.

\end{document}